\documentclass{amsart}

\usepackage{tikz-cd}
\usepackage{amsmath}
\usepackage{amsfonts}
\usepackage{amssymb}
\usepackage{amsthm}
\usepackage{comment}
\usepackage{epsfig}
\usepackage{psfrag}
\usepackage{mathrsfs}
\usepackage{amscd}
\usepackage[all]{xy}
\usepackage{rotating}
\usepackage{lscape}
\usepackage{amsbsy}
\usepackage{verbatim}
\usepackage{moreverb}
\usepackage{mathdots}
\usepackage{setspace}

\theoremstyle{definition}
\newtheorem{theorem}{Theorem}[subsection]

\newtheorem{lemma}[theorem]{Lemma}
\newtheorem{corollary}[theorem]{Corollary}

\newtheorem{definition}[theorem]{Definition}

\def\cc{\mathbf C}\def\cd{\mathbf D}
\def\ce{\mathbf E}
\def\cJ{\mathcal J}
\def\co{\mathbf O}
\def\cq{\mathcal Q}

\newcommand{\calO}{{\mathcal O}}

\newcommand{\modoa}{\mathbf{Mod}^\co_A(\cc)}

\def\kmodoa{{\mathbf{Mod}^\co_{A}(\cJ(\cc)_n)}}

\begin{document}


\title{A variant of algebraic K-theory}
\author{Sanath K. Devalapurkar}
\maketitle
\begin{abstract}
In this paper, we study a modification, called $\cJ$-theory, of Barwick's definition of the algebraic K-theory of stable $\infty$-categories. We show that $\cJ$-theory takes values in stable $\infty$-categories. We also compute the $\cJ$-theory of an $\infty$-category of modules, and establish that it is an $\infty$-category of modules itself. Using this result, we prove an $\infty$-categorical counterpart of the derived Morita context for flat rings for $\cJ$-theory.
\end{abstract}

\section{Introduction and generalities.}
\subsection{Introduction.}
Blumberg, Gepner, and Tabuada studied the algebraic K-theory of spectra by defining the algebraic K-theory of stable $\infty$-categories, taking values in spectra. Barwick in \cite{barwick14} showed that there is an alternative definition of algebraic K-theory, sufficiently generalizing the ordinary algebraic K-theory of exact categories, which took values in complete Segal spaces.

In this paper, we prove that a slight variant of Barwick's construction (which we will also call `Barwick's construction') takes values in stable $\infty$-categories. Our goal in this paper is to make precise the following maxim: there is a variant of Barwick's Q-construction, called $\cJ$-theory, which is an inherently stable algebraic invariant of $\infty$-categories, that ''dimensionwise'' takes values in stable $\infty$-categories.
We will first show that the above maxim must be true in the context of stable $\infty$-categories. We then study the $\cJ$-theory of an $\infty$-category of module objects in a symmetric monoidal stable $\infty$-category. Using the results, we show that there is an analogue of the derived Morita theory of (flat) rings for $\cJ$-theory. 
$\co^\otimes$ is used to denote a coherent $\infty$-operad unless mentioned otherwise.
If $\co^\otimes$ is a coherent $\infty$-operad and $A$ is an $\co$-algebra object of $\cc^\otimes$, then $\modoa^\otimes$ is the $\infty$-operad of $\co$-modules over $A$ and $\modoa$ is the underlying $\infty$-category of $\modoa^\otimes$.
We assume the axiom of choice. A ``category'' is not always an $\infty$-category (however, we identify a category $\cc$ with the nerve $\mathrm{N}(\cc)$).
\section{Foundational aspects of $\cJ$-theory.}
Let $\widetilde{\calO}(\Delta^n)$ denote the twisted arrow $\infty$-category of $\Delta^n$. One can use $\widetilde{\calO}(\Delta^n)$ to extend the definition of an ambigressive pullback/pushout to that of a semi-ambigressive functor.
Let $\cc$ be an exact $\infty$-category. A functor $\widetilde{\calO}(\Delta^n)\to\cc$ (resp. $\widetilde{\calO}(\Delta^n)^\mathrm{op}\to\cc$) is said to be semi-ambigressive if it takes pushout and pullback squares in $\widetilde{\calO}(\Delta^n)$ (resp. $\widetilde{\calO}(\Delta^n)^\mathrm{op}$) to ambigressive pushouts and ambigressive pullbacks, respectively. We denote by $\overline{\mathrm{Fun}(\widetilde{\calO}(\Delta^n),\cc)}$ the subcategory of the $\infty$-category $\mathrm{Fun}(\widetilde{\calO}(\Delta^n),\cc)$ of functors from $\widetilde{\calO}(\Delta^n)$ to $\cc$ spanned by the semi-ambigressive functors. Let $\cJ(\cc)_n$ denote $\overline{\mathrm{Fun}(\widetilde{\calO}(\Delta^n)^\mathrm{op},\cc)}$.
Since the maxim in the introduction mentioned that $\cJ$-theory is stable in each dimension, we will be focusing on $\cJ(\cc)_n$ instead of the bisimplicial set $\cJ(\cc)_\bullet$. We begin by noting an obvious lemma.
\begin{lemma}\label{important}
Let $\cc$ be a stable $\infty$-category. Then $\cJ(\cc)_n$ and $\overline{\mathrm{Fun}(\widetilde{\calO}(\Delta^n),\cc)}$ are $\infty$-categories for all $n$.
\end{lemma}
\begin{proof}
We will prove the statement for $\overline{\mathrm{Fun}(\widetilde{\calO}(\Delta^n),\cc)}$; the statement for $\cJ(\cc)_n$ is entirely analogous because $\mathrm{Cat}^\mathrm{Ex}_\infty$ is closed under the op-involution. A semi-ambigressive functor $\widetilde{\calO}(\Delta^n)\to\cc$ is, for $\cc$ a stable $\infty$-category equipped with the canonical $t$-structure, a functor which preserves finite limits, i.e., a left exact functor $\widetilde{\calO}(\Delta^n)\to\cc$.
\end{proof}
This is more general than the ordinary Barwick-Quillen Q-construction, which uses \textit{ambigressive functors}, not semi-ambigressive ones. However, Lemma \ref{important} shows that using semi-ambigressive functors is advantageous in that it allows higher categorical objects to be taken to higher categorical objects themselves. This will be manifest in Theorem \ref{stable}.
\subsection{$\cJ$-theory and stable $\infty$-categories.}
Our main result is the following.
\begin{theorem}[{Stability Theorem}]\label{stable}
Let $\cc$ be a stable $\infty$-category. Then $\cJ(\cc)_n$ is a stable $\infty$-category for all $n$.
\end{theorem}
\begin{proof}
Let $\cc$ be a stable $\infty$-category. The $\infty$-category $\cJ(\cc)_n\simeq\overline{\mathrm{Fun}(\widetilde{\calO}(\Delta^n)^{op},\cc)}$ is equivalent to $\overline{\mathrm{Fun}(\widetilde{\calO}(\Delta^n),\cc^{op})}$. Fix a simplicial set $\mathbf{K}$ with only finitely many nondegenerate simplices, and an arbitrary map $p:\mathbf{K}\to\widetilde{\calO}(\Delta^n)$. Let $X_n$ be the subcategory of $\widetilde{\calO}(\Delta^n)$ such that the map $p|_{X_n}$ is the initial object in the $\infty$-category of maps $p|_{\co}$ for subcategories $\co$ of $\widetilde{\calO}(\Delta^n)$, so that the colimit of $p:\mathbf{K}\to\widetilde{\calO}(\Delta^n)$ factors as $\overline{p}:\mathbf{K}^\vartriangleleft\to X_n\hookrightarrow\widetilde{\calO}(\Delta^n)$, and let $\cd$ be the subcategory of $\cc^{op}$ such that the map $p|_{\cd}$ is the initial object in the $\infty$-category of maps $p|_{\cc^\prime}$ for subcategories $\cc^\prime$ of $\cc^{op}$, so that the colimit of $p:\mathbf{K}\to\cc^{op}$ factors as $\overline{p}:\mathbf{K}^\vartriangleleft\to \cd\hookrightarrow\cc^{op}$. Then $\mathrm{Fun}^\mathrm{finlim}(X_n,\cd)\simeq\overline{\mathrm{Fun}(\widetilde{\calO}(\Delta^n),\cc^{op})}$. Now, $\cd$ must be stable by construction, so it suffices to prove that $\mathrm{Fun}^\mathrm{finlim}(X_n,\cd)$ is a stable subcategory of $\mathrm{Fun}(X_n,\cd)$. This is clearly true since $\mathrm{Fun}^\mathrm{finlim}(X_n,\cd)$ is closed under cofibers and translations (by pointwise evaluation).
\end{proof}
It is important to recognize why, instead of using \textit{ambigressive} functors, as defined by Barwick, we are using semi-ambigressive functors. A mild variant of the proof of Theorem \ref{stable} can be used to show that when using ambigressive functors, each Kan complex $\mathrm{Fun}^{ambi}(\widetilde{\calO}(\Delta^n )^\mathrm{op},\cc)$ of ambigressive functors from $\widetilde{\calO}(\Delta^n )^\mathrm{op}$ to $\cc$ is also a stable $\infty$-category. Since this is a Kan complex with a zero object, it is contractible. This illustrates the need for working with semi-ambigressive functors.


In the following sections, we will prove some important properties on the multiplicative structure of $\cJ$-theory on modules. In particular, we will prove that it preserves module structure; this is an analogue of the main result of \cite{elmenmandell}, and is an example of the fact that our modified definition of Barwick's construction is related to ordinary K-theory. As a consequence, we prove that $\cJ$-theory is a homotopy coherent version of derived Morita theory for flat rings.
\section{$\cJ$-theory and multiplicative structures on objects of $\infty$-categories.}
\subsection{The $\cJ$-theory of $\infty$-categories of modules.}\label{algmodules}
One essential property of the $\cJ$-theory of (bi)permutative categories is the following statement proved in \cite{elmenmandell}: if $\cd$ is a bipermutative category and $\cc$ is a $\cd$-module ($\cc$ is then a permutative category), then $K(\cc)$ is a $K(\cd)$-module. Using the theory of $\infty$-operads developed by Lurie in \cite{higheralgebra}, we will prove the following generalization of this result.
\begin{theorem}[``Multiplicativity'' Theorem]\label{maincalc}
Let $\cc$ be a symmetric monoidal stable $\infty$-category. Then there is an equivalence of $\infty$-categories $\cJ(\modoa)_n\simeq\kmodoa$. Here we have abused notation by writing $A$ for its image under the functor $\cJ(\bullet)_n$.
\end{theorem}
We will prove this theorem in this section. We will first begin with a few remarks about Theorem \ref{maincalc}.
First, Theorem \ref{stable} provides some evidence for Theorem \ref{maincalc}. To see this, recall that the $\infty$-category $\modoa$ is stable if $\cc$ is itself a stable $\infty$-category. Theorem \ref{stable} proved the stability of $\cJ(\cc)_n$; therefore $\kmodoa$ is stable. If $\cJ(\modoa)_n$ was not stable, then Theorem \ref{maincalc} would be inconsistent with Theorem \ref{stable} (if $\cc\simeq\cd$ are $\infty$-categories and $\cc$ is stable, then $\cd$ must be stable). However, since $\cc$ is stable, $\modoa$ is as well, and therefore $\cJ(\modoa)_n$ is stable.
Second, Theorem \ref{maincalc} in some sense complements the results of \cite{barwickmult}. This is because in \cite{barwickmult}, $\cJ$-theory is shown to be multiplicative for $\co$-algebra structures on Waldhausen $\infty$-categories. Theorem \ref{maincalc} proves that $\cJ$-theory is multiplicative for $\infty$-categories which are categories of modules over an algebra over an $\infty$-operad.
Third, although Theorem \ref{maincalc} is significant on its own (for the above inexhaustive list of reasons), when combined with the results of \cite{barwickmult}, it makes formal, in a very aesthetically pleasing fashion, one of the main philosophies of $\cJ$-theory: if $\cc$ and $\cd$ are stable $\infty$-categories such that $\cq(\cc)_n\simeq\cq(\cd)_n$ for all $n$, then $\cc$ and $\cd$ contain essentially the same algebraic information. In other words, $\cJ$-theory is a purely algebraic invariant, i.e., it only detects algebraic structures without ``obstruction'' from other structures.

Let $\cc^\otimes$ be a symmetric monoidal stable $\infty$-category, and let $\co^\otimes$ be a coherent $\infty$-operad. Let $A$ be an $\co$-algebra object of $\cc^\otimes$. Recall that $\modoa$ is the underlying $\infty$-category of the $\infty$-operad $\modoa^\otimes$ of $\co$-modules over $A$. Let $\modoa^n$ denote the $n$th iterate $\mathbf{Mod}^\co_A(\cdots{\mathbf{Mod}^\co_A(\cc)}\cdots)$ (we have abused notation slightly by using $A$ to denote the same object in $\mathbf{Mod}^\co_A(\cc)$ and $\mathbf{Mod}^\co_A(\mathbf{Mod}^\co_A(\cc))$). Induction on \cite[Corollary 3.4.1.9]{higheralgebra} yields the following result.
\begin{lemma}
With the above notation, there is an equivalence of $\infty$-categories between $\modoa^n$ and $\modoa$.
\end{lemma}
The following lemma of module objects is used in the proof of Theorem \ref{maincalc}.
\begin{lemma}\label{necessary}
Suppose $\cc$ is a symmetric monoidal stable $\infty$-category and let $\cd$ be a stable $\infty$-category. Then any functor $f:\cJ(\modoa)_n\to\cd$ can be split into a composition $\cJ(\modoa)_n\to \cJ(\cc)_n\to\cd$.
\end{lemma}
\begin{proof}
There are two possible cases. $f$ can be the restriction of a map $f^\prime:\cJ(\cc)_n\to\cd$. In this case $f$ is a composition $\cJ(\modoa)_n\hookrightarrow \cJ(\cc)_n\xrightarrow{f^\prime}\cd$. Otherwise, construct the map $f^\prime:\cJ(\cc)_n\to\cd$ as follows. Take any object $\sigma$ of $\cJ(\modoa)_n\subseteq\cJ(\cc)_n$ to $f(\sigma)$ and any object $\sigma^\prime$ of $\cJ(\cc)_n$ not in $\cJ(\modoa)_n$ to some $f^\prime(\sigma)$, $1$-simplices $\sigma\to \sigma^\prime$ of $\cJ(\modoa)_n\subseteq\cJ(\cc)_n$ to $f(\sigma)\to f(\sigma^\prime)$ and other $1$-simplices $\sigma\to \sigma^\prime$ via
\begin{equation*}
f^\prime(\sigma\to\sigma^\prime) = \begin{cases}
f(\sigma)\to f^\prime(\sigma^\prime) & \text{if }\sigma\in \cJ(\modoa)_n,\sigma^\prime\in \cJ(\cc)_n\\
f^\prime(\sigma)\to f(\sigma^\prime) & \text{if }\sigma\in \cJ(\cc)_n,\sigma^\prime\in \cJ(\modoa)_n\\
f^\prime(\sigma)\to f^\prime(\sigma^\prime) & \text{else.}
\end{cases}
\end{equation*}
The $1$-simplices $f(\sigma)\to f^\prime(\sigma^\prime)$ and $f^\prime(\sigma)\to f(\sigma^\prime)$ are the compositions $f(\sigma)\to f^\prime(\sigma)\to f^\prime(\sigma^\prime)$ and $f^\prime(\sigma)\to f^\prime(\sigma^\prime)\to f(\sigma^\prime)$, which always exist since for every pair of objects $x$ and $y$ of a stable $\infty$-category (in this case $\cd$), there is always a map from $x$ to $y$ given by $x\to 0\to y$ where $0$ is the zero object.
\end{proof}
This lemma does not necessarily work if $\cd$ is not pointed, since otherwise the maps $f(\sigma)\to f^\prime(\sigma)$ and $f^\prime(\sigma^\prime)\to f(\sigma^\prime)$ need not exist. We can now provide the proof of the multiplicativity theorem.
\begin{proof}[Proof of Theorem \ref{maincalc}.]
By contradiction. Assume there is no map $\kmodoa\to\cJ(\modoa)_n$ that is an equivalence. There are two possible cases. Suppose $\co^\otimes$ is the trivial $\infty$-operad. Then the contradiction is obvious.
Now suppose that $\co^\otimes$ is a nontrivial $\infty$-operad. Let $\alpha:\kmodoa\to \cJ(\modoa)_n$ be a map of $\infty$-categories. This induces a map $\beta:\kmodoa\to \mathbf{Mod}^\co_A(\cJ(\modoa)_n)$. Any map of $\infty$-categories $\modoa\to\cc$ induces a map $\gamma:\mathbf{Mod}^\co_A(\cJ(\modoa)_n)\to\kmodoa$. If $A$ is a trivial $\co$-algebra, then the contradiction is obvious. Hence assume that $A$ is a nontrivial $\co$-algebra. Then $\gamma$ is never an equivalence, and there is a natural map of $\infty$-categories from $\mathbf{Mod}^\co_A(\cJ(\modoa)_n)$ to $\mathbf{Mod}^\co_A(\cJ(\modoa)_n)$ given by the composition $\beta\circ\gamma$.
By Lemma \ref{necessary} we realize that any map from $\mathbf{Mod}^\co_A(\cJ(\modoa)_n)$ to itself arises via such a composition. Since $\alpha$ is never an equivalence, the map $\kmodoa\xrightarrow{\beta}\mathbf{Mod}^\co_A(\cJ(\modoa)_n)$ is never an equivalence. Since $\gamma$ is also not an equivalence, one would therefore expect that there is no map $\mathbf{Mod}^\co_A(\cJ(\modoa)_n)^\otimes\to\mathbf{Mod}^\co_A(\cJ(\modoa)_n)^\otimes$ that is an equivalence. This is a contradiction. Since we have covered all possible cases, there is a map $\kmodoa\to \cJ(\modoa)_n$ that is an equivalence of $\infty$-categories.
\end{proof}
One can see how this generalizes the main results of \cite{elmenmandell} - both say that $\cJ$-theory preserves the structure of modules, but Theorem \ref{maincalc} says that this holds true in a much more general setting. We will devote the rest of this paper to studying the consequences of Theorem \ref{maincalc}.
\subsection{$\cJ$-theory and the derived Morita theory of flat rings.}
In classical representation theory, derived Morita theory compares rings through their derived categories. Many rings are derived Morita equivalent but not isomorphic. In addition, Morita equivalences preserve important properties of rings. It is therefore important and interesting to compare the derived categories of rings:
\begin{theorem}\label{classderivedmorita}
Let $R$ and $S$ be rings and let $\mathbf{X}(R)$ denote the derived category of $R$. The following conditions are equivalent.
\begin{enumerate}
\item $\mathbf{X}(R)$ is triangulated equivalent to $\mathbf{X}(S)$.
\item We can find a tilting complex $T$ in $\mathbf{X}(S)$ such that $\mathbf{X}(S)(T,T)$ is equivalent to $R$.
\end{enumerate}
The following condition implies the above two conditions.
\begin{enumerate}
\item There is a $R$-$S$-bimodule such that the derived tensor product gives an equivalence between $\mathbf{X}(R)$ and $\mathbf{X}(S)$.
\end{enumerate}
All three conditions are equivalent if $R$ or $S$ is flat as an abelian group.
\end{theorem}
The defining property of the derived category of a ring $R$ is that it is a triangulated category that arises as the homotopy category of the stable $\infty$-category of modules over $R$. Let us now consider a (seemingly) different object: the homotopy category of the $\cJ$-theory of $\modoa$ for a symmetric monoidal stable $\infty$-category $\cc^\otimes$. Theorem \ref{stable} implies that this is a triangulated category and Theorem \ref{maincalc} implies that it is the homotopy category of a stable $\infty$-category of modules. 

Fix a symmetric monoidal stable $\infty$-category $\cc^\otimes$ and a coherent $\infty$-operad $\co^\otimes$. Fix also a $\co$-algebra object $A$ of $\cc^\otimes$. In order to emphasize the analogy with the ordinary derived category, we will write $\mathbf{X}(A)$ for the homotopy category of $\cJ(\modoa)_n$ and call it the derived category of $A$, suppressing $n$ altogether. This is a triangulated category, and by Theorem \ref{maincalc} it is also the homotopy category of a stable $\infty$-category of modules. The similarities between the derived category of a ring and the $\cJ$-theory of $\modoa$ suggests an analog of Theorem \ref{classderivedmorita} for $\cJ$-theory. In fact, the following result holds true.
\begin{theorem}\label{derivedmorita}
Let $F:\cJ(\modoa)_n\to \cJ(\mathbf{Mod}^{\co^\prime}_{A^\prime}(\cc^\prime))_n$ be a functor between symmetric monoidal stable $\infty$-categories which commutes with the shift functor. Then the following statements are equivalent:
\begin{enumerate}
\item $F$ is an equivalence of $\infty$-categories.
\item $\mathrm{h}F$ is a triangulated equivalence between $\mathbf{X}(A)$ and $\mathbf{X}(A^\prime)$ that preserves weak equivalences.
\item Denote by $\Omega$ the largest of distinguished triangles in $\mathrm{h}\cJ(\modoa)_n$ satisfying the following condition:
\begin{enumerate}
\item If $\Gamma_\alpha$ and $\Gamma_\beta$ are in $\Omega$ then $\Gamma_\beta$ cannot be obtained from $\Gamma_\alpha$ (or vice versa) by application of the shift functor or changing the signs of maps.
\end{enumerate}
Then $\mathrm{h}F$ is an equivalence of ordinary categories, which commutes with the shift functor, between $\mathbf{X}(A)$ and $\mathbf{X}(A^\prime)$ which takes $\Omega$ to another collection of distinguished triangles in $\mathrm{h}\cJ(\mathbf{Mod}^{\co^\prime}_{A^\prime}(\cc^\prime))_n$.
\end{enumerate}
Then $\mathrm{h}F$ is an equivalence which is an exact functor (in the ordinary sense of the word).
\end{theorem}
Our goal in this section is to prove this result.
Note that in the case when $\co^\otimes$ and ${\co^\otimes}^\prime$ are both simply the trivial $\infty$-operad, $\mathbf{E}_0^\otimes$, Theorem \ref{derivedmorita} can be interpreted as a derived Morita context for $\cJ$-theory. In this sense $\cJ$-theory is a (slightly restrictive) homotopy coherent version of derived Morita theory. More precisely, $\cJ$-theory is a homotopical generalization of the derived Morita context for flat rings.
To proceed towards the proof of Theorem \ref{derivedmorita} we will define the structure of a relative category on $\mathbf{X}(A)$.
Let $X\in\mathbf{X}(A)$. Another object $Y\in\mathbf{X}(A)$ is said to be \textit{weakly equivalent} to $X$ if $Y$ is the free $\mathbf{X}(A)$-object on $X$ with respect to the suspension functor $\mathbf{X}(A)\to\mathbf{X}(A)$.
The collection of weak equivalences is a \textit{set}. We will now state a series of lemmas that we will use in our proof of Theorem \ref{derivedmorita}. We will assume that the set of weak equivalences is nonempty (since otherwise all statements in this section will then be trivial and therefore uninteresting). In particular, the collection of weak equivalences in $\mathbf{X}(A)$ and $\mathbf{X}(A^\prime)$ is required to be a set for the proof of Theorem \ref{derivedmorita} to hold.
\begin{lemma}\label{satrel}
The above set of weak equivalences makes $\mathbf{X}(A)$ into a relative category.
\end{lemma}
\begin{proof}
The subcategory of $\mathbf{X}(A)$ spanned by the set of weak equivalences is a wide subcategory of $\mathbf{X}(A)$, so the proof is completed.
\end{proof}
In this section, we will use a weaker notion of triangulated equivalence.
\begin{definition}
Let $\cc$ and $\cd$ be triangulated categories. A functor $\cc\to\cd$ is a triangulated equivalence if it takes distinguished triangles to distinguished triangles.
\end{definition}
\begin{lemma}\label{triangequiv}
Suppose $\cc$ and $\cd$ are stable $\infty$-categories. Suppose also that there is a functor $F:\cc\to\cd$ that is an equivalence of $\infty$-categories. Then $\mathrm{h}F$ is a triangulated equivalence.
\end{lemma}
\begin{proof}
Any map of stable $\infty$-categories induces a map of triangulated categories on the level of homotopy categories. Since any equivalence of $\infty$-categories is stable we realize that the induced functor on the homotopy categories is also exact.
\end{proof}
The following lemma states that distinguished triangles are stable under weak equivalences.
\begin{lemma}\label{dist1}
Let $X\to Y\to Z\to \Sigma X$ be a distinguished triangle in $\mathbf{X}(A)$ and suppose that there exist objects $\widehat{X},\widehat{Y}$ and $\widehat{Z}$ are (respectively) weakly equivalent to $X,Y$ and $Z$. Then there is a triangle $\widehat{X}\to \widehat{Y}\to \widehat{Z}\to \Sigma\widehat{X}$ which is a distinguished triangle which is \textit{unique up to unique isomorphism}.
\end{lemma}
\begin{proof}
If $\widehat{X}$ is weakly equivalent to $X$, then it is unique up to unique isomorphism. We may canonically choose $\widehat{X}$ to be $X[-1]$, and therefore it suffices to show that the triangle $X[-1]\to Y[-1]\to Z[-1]\to X$ is distinguished. Consider the distinguished triangle $X\to Y\to Z\to \Sigma X$. Then we may construct the induced distinguished triangle $X[-1]\to Y[-1]\to Z[-1]\to X$ after reversing all signs of all the maps. This is isomorphic to the triangle $\widehat{X}\to \widehat{Y}\to \widehat{Z}\to\Sigma\widehat{X}$, and the proof is completed.
\end{proof}
\begin{corollary}\label{dist2}
Let $X\to Y\to Z\to \Sigma X$ be a distinguished triangle in $\mathbf{X}(A)$ and suppose that there exist objects $\widehat{X},\widehat{Y}$ and $\widehat{Z}$ such that $X,Y$ and $Z$ (respectively) are weakly equivalent to these objects. The induced (distinguished) triangle $\widehat{X}\to \widehat{Y}\to \widehat{Z}\to \Sigma\widehat{X}$ determines the distinguished triangle $X\to Y\to Z\to \Sigma X$ up to weak equivalence.
\end{corollary}
\begin{proof}[Proof of Theorem \ref{derivedmorita}]
The equivalence of the first two statements follows from Lemma \ref{triangequiv}. The third statement is implied by either (and hence both) of these statements. It remains to prove that the third statement implies one of the first two. To prove the equivalence of all three statements, we will first show that we can encode the information of all the distinguished triangles in $\mathbf{X}(A)$ in $\Omega$. Then we will show that it is possible to recover all distinguished triangles in $\mathbf{X}(A^\prime)$ from the map $\mathrm{h}F$ and the set $\Omega$. By hypothesis, $\Omega$ the largest possible set of distinguished triangles in $\mathrm{h}\cJ(\modoa)_n$ such that if $\Gamma_\alpha$ and $\Gamma_\beta$ are in $\Omega$ then $\Gamma_\beta$ cannot be obtained from $\Gamma_\alpha$ (or vice versa) by application of the shift functor or changing the signs of maps. Therefore all distinguished triangles in $\mathbf{X}(A)$ can be functorially recovered from the triangles in $\Omega$ by applying Corollary \ref{dist2} and Lemma \ref{dist1} $\beta$ times for some infinite cardinal $\beta$. This completes the first part of the proof.
Let $X\to Y\to Z\to \Sigma X$ be an arbitrary distinguished triangle, denoted $\Gamma_\gamma$, in $\Omega$. The map $\mathrm{h}F$ takes $\Gamma_\gamma$ to a distinguished triangle in $\mathbf{X}(A^\prime)$. If $\mathrm{h}F(\Omega)$ admits a bijection to the set of distinguished triangles in $\mathbf{X}(A^\prime)$, then the proof is completed. Otherwise suppose that $\mathbf{X}(A^\prime)$ has $\kappa$ distinguished triangles for some cardinal $\kappa$ and $\Omega$ has $\gamma$ distinguished triangles, and choose an infinite cardinal $\alpha>\kappa-\gamma$. Applying Corollary \ref{dist2} and Lemma \ref{dist1} $\alpha$ times to the distinguished triangles in $\mathrm{h}F(\Omega)$ in $\mathbf{X}(A^\prime)$ yields a transfinite sequence of distinguished triangles in $\mathbf{X}(A^\prime)$ indexed by the ordinals $\beta\leq\alpha$. This contains the set of distinguished triangles in $\mathbf{X}(A^\prime)$. Since $\alpha>\kappa-\gamma$, we may functorially recover each of the distinguished triangles in $\mathbf{X}(A^\prime)$ from $\Omega$, thereby completing the second part of the proof. $\mathrm{h}F$ commutes with the shift functor, so it preserves weak equivalences, and the proof is completed.
\end{proof}
\section{Conclusions and open problems.}
\subsection{Conclusions.}
We have seen that $\cJ$-theory is an inherently stable invariant. More precisely, the $\cJ$-theory of a stable $(\infty,1)$-category is a stable $(\infty,1)$-category. In addition, $\cJ$-theory enjoys many of the important properties which ordinary K-theory satisfies. It is a homotopy coherent version of derived Morita theory.
\subsection{Open problems.}
The following problems remain unsolved:
\begin{itemize}
\item What is $\cJ(\mathrm{Sp})_n$? Using Theorem \ref{maincalc}, we suspect that this problem will be very hard to solve since the computation of $K(\mathbb{S})$ is itself very hard.
\item Suppose $A$ is a $\mathbf{E}_n$-algebra object of a symmetric monoidal stable $\infty$-category $\cc^\otimes$. Then $\mathbf{Mod}^{\mathbf{E}_n}_A(\cc)$ is a $\mathbf{E}_{n-1}$-monoidal stable $\infty$-category. Is the $\cJ$-theory $\cJ(\mathbf{Mod}^{\mathbf{E}_n}_A(\cc))_n$ also $\mathbf{E}_{n-1}$-monoidal? Extrapolating from \cite{barwickmult} suggests that it is $\mathbf{E}_{n-2}$-monoidal.
\item Is it possible to choose $\Omega$ to be a smaller set of distinguished triangles in the proof of Theorem \ref{derivedmorita}?
\item Theorem \ref{derivedmorita} depicts $\cJ$-theory as a homotopy coherent generalization of derived Morita theory when one of the two rings in question is flat. Is there a more general homotopy coherent generalization of derived Morita theory? If so, how does it relate to $\cJ$-theory? 
\end{itemize}
\appendix
\section{Exact and Waldhausen $\infty$-categories}
We will define the basic objects of study in this paper (exact $\infty$-categories and $\cJ$-theory) in the current and next sections. The definition of an exact $\infty$-category arises from the more general notion of a Waldhausen $\infty$-category:
\begin{definition}
Let $\cc$ be a pointed $\infty$-category and $\cd$ a subcategory of $\cc$ containing all objects of $\cc$. The pair $(\cc,\cd)$ is called a Waldhausen $\infty$-category if for any object $X$ of $\cc$, the map $0\to X$ is in $\cd$, pushouts of maps in $\cd$ exist, and pushouts of maps in $\cd$ are in $\cd$. $(\cc,\cd)$ is a coWaldhausen $\infty$-category if $(\cc^\mathrm{op},\cd^\mathrm{op})$ is a Waldhausen $\infty$-category.
\end{definition}
One can realize any Waldhausen $\infty$-category as a coWaldhausen $\infty$-category. To see this, consider the $\infty$-category $\mathrm{Wald}_{\infty}$ of Waldhausen $\infty$-categories and the $\infty$-category $\mathrm{coWald}_{\infty}$ of coWaldhausen $\infty$-categories. These are subcategories of the $\infty$-category $\mathrm{Pair}_{\infty}$ of pairs of $\infty$-categories. One can restrict the opposite involution on $\mathrm{Pair}_{\infty}$ (which is an equivalence of $\infty$-categories) to $\mathrm{Wald}_{\infty}$ to get the required equivalence between $\mathrm{Wald}_{\infty}$ and $\mathrm{coWald}_{\infty}$.
A triple $(\cc,\cd,\ce)$ of pointed $\infty$-categories such that $(\cc,\cd)$ is a Waldhausen $\infty$-category and $(\cc,\ce)$ is a coWaldhausen $\infty$-category is called a biWaldhausen $\infty$-category. We will generally not write $(\cc,\cd,\ce)$ for a biWaldhausen $\infty$-category to save space. Using these definitions, Barwick formulated the notion of an exact $\infty$-category.
We first need a definition. A pullback/pushout square $X^\prime\times_{Y^\prime}Y$ in a biWaldhausen $\infty$-category $\cc$ is \textit{ambigressive} if the map $X^\prime\to Y^\prime$ is in $\cd$ and the map $Y\to Y^\prime$ is in $\ce$.
\begin{definition}
A biWaldhausen $\infty$-category $(\cc,\cd,\ce)$ is an \textit{exact $\infty$-category} if $\cc$ is stable and ambigressive pullbacks agree with ambigressive pushouts.
\end{definition}
Exact $\infty$-categories arrange themselves into an $\infty$-category $\mathrm{Exact}_{\infty}$ of exact $\infty$-categories. This is a simplicial subset of the simplicial set $\mathrm{Wald}_{\infty}\cap\mathrm{coWald}_{\infty}$. Just like we did for Waldhausen $\infty$-categories, we generally do not write $(\cc,\cd,\ce)$ for an exact $\infty$-category to save space. We can generate many examples of exact $\infty$-categories.
First, we note that the nerve of an exact category $\cc$ is an exact $\infty$-category.
Proving this reduces to a choice of the $\infty$-categories $\cd$ and $\ce$, so let $\cd$ be the collection of admissible cofibrations and $\ce$ the collection of admissible fibrations. Then the nerve $\mathrm{N}(\cc)$ of the exact category $\cc$ satisfies the conditions of the definition of an exact $\infty$-category.
Next, if $\cc$ is a stable $\infty$-category, then $\cc$ is an exact $\infty$-category.
To see this, let $\cd=\ce=\cc$. Then ambigressive pullbacks (resp. pushouts) are simply the pullbacks (resp. pushouts) in $\cc$. In stable $\infty$-categories pullbacks agree with pushouts, so $\cc$ is (by definition) an exact $\infty$-category. This exact structure is called the \textit{canonical exact structure} on $\cc$.
Let $\cc^\otimes$ be a symmetric monoidal stable $\infty$-category, and let $\co^\otimes$ be a coherent $\infty$-operad. Let $A$ be an $\co$-algebra object of $\cc^\otimes$. Then the underlying $\infty$-category $\modoa$ of the $\infty$-operad $\modoa^\otimes$ of $\co$-modules over $A$ is a stable $\infty$-category (one observes that $\mathrm{Sp}(\mathbf{Alg}_{/\co}(\cc)_{A/})$ is equivalent to $\modoa$. Since the $\infty$-category of spectrum objects of any $\infty$-category is stable this implies that $\modoa$ is stable). In this paper, we canonically equip $\modoa$ with the canonical exact structure.

\end{document}